\documentclass[12pt,a4paper]{article}
\usepackage{amssymb,amsmath,amsthm}

\newcommand\cF{{\mathcal F}}

\newcommand\cR{{\mathcal R}}

\newcommand\cU{{\mathcal U}}

\theoremstyle{plain}
\newtheorem{theorem}{Theorem}[section]

\newtheorem{conjecture}[theorem]{Conjecture}

\theoremstyle{definition}

\newtheorem{claim}[theorem]{Claim}

\newcommand\tref[1]{Theorem~\ref{thm:#1}}
\newcommand\cref[1]{Corollary~\ref{cor:#1}}
\newcommand\clref[1]{Claim~\ref{clm:#1}}

\newcommand\sref[1]{Section~\ref{sec:#1}}

\begin{document}

\title{On the ratio of maximum and minimum degree in maximal intersecting families}

\author{Zolt\'an L\'or\'ant Nagy\thanks{Alfr\'ed R\'enyi Institute of Mathematics, P.O.B. 127, Budapest H-1364, Hungary and E\"{o}tv\"{o}s Lor\'{a}nd University, Department of Computer Science, H-1117 Budapest P\'{a}zm\'{a}ny P\'{e}ter s\'{e}t\'{a}ny 1/C. Email: nagyzoltanlorant@gmail.com. The author was supported by the Hungarian National Foundation for Scientific Research (OTKA), Grant no. K 81310.} \and Lale \"Ozkahya\thanks{Email: ozkahya@illinoisalumni.org} \and Bal\'azs Patk\'os\thanks{Alfr\'ed R\'enyi Institute of Mathematics, P.O.B. 127, Budapest H-1364, Hungary. Email: patkos@renyi.hu. Research is supported by OTKA Grant PD-83586 and
 the J\'anos Bolyai Research Scholarship of the Hungarian Academy of Sciences.} \and M\'at\'e Vizer\thanks{Central European University, Department of Mathematics and its Applications, Budapest, N\'ador u. 9. H-1051, Hungary. Email: vizermate@gmail.com.}}

\maketitle

\begin{abstract}
To study how balanced or unbalanced a maximal intersecting family $\cF\subseteq \binom{[n]}{r}$ is we consider the ratio $\cR(\cF)=\frac{\Delta(\cF)}{\delta(\cF)}$ of its maximum and minimum degree. We determine the order of magnitude of the function $m(n,r)$, the minimum possible value of $\cR(\cF)$, and establish some lower and upper bounds on the function $M(n,r)$, the maximum possible value of $\cR(\cF)$. To obtain constructions that show the bounds on $m(n,r)$ we use a theorem of Blokhuis on the minimum size of a non-trivial blocking set in projective planes.
\end{abstract}

\textit{Keywords}: Intersecting families, maximum and minimum degree, blocking sets

\textit{AMS subject classification}: 05D05, 05B25

\section{Introduction}
A family $\cF$ of sets is said to be intersecting if $F_1 \cap F_2 \ne \emptyset$ holds for any $F_1,F_2 \in \cF$. In their
seminal paper, Erd\H os, Ko and Rado showed \cite{EKR} that if $\cF$ is an intersecting family of $r$-subsets of an $n$-element set $X$ (we denote this by $\cF\subseteq \binom{X}{r}$), then $|\cF| \le\binom{n-1}{r-1}$ provided that $2r \le n$. Lots of generalizations of the above theorem have been considered ever since and lots of researchers have been interested in describing how intersecting families may look like. One of the quantities concerning intersecting families that has been studied \cite{DF, LP} is the unbalance $\cU(\cF)=|\cF|-\Delta(\cF)$ where $\Delta(\cF)$ denotes the maximum degree in $\cF$. In this paper we define another notion to measure the balancedness or unbalancedness of $\cF$: if $\delta(\cF)$ denotes the minimum degree in $\cF$, then our aim is to determine how small and how large $\cR(\cF)=\frac{\Delta(\cF)}{\delta(\cF)}$ can be. To avoid $\delta(\cF)=0$ we will always assume that $\cup \cF=X$. Also, as by considering appropriate subfamilies one could modify the value $\cR(\cF)$ easily, we will restrict our attention to \textit{maximal} intersecting families, i.e. families with the property $G \in \binom{X}{r}\setminus \cF \Rightarrow \exists F \in \cF \hskip 0.2truecm F \cap G =\emptyset$. For sake of simplicity we will also assume that the underlying set $X$ of our families is $[n]=\{1,2,...,n\}$.

With the above notation and motivation we define our two main functions as follows:
\[
M(n,r)=\max\left\{\cR(\cF): \hskip 0.2truecm\cF \subseteq \binom{[n]}{r} \text{is maximal intersecting with}\ \cup \cF=[n]\right\},
\]
\[
m(n,r)=\min\left\{\cR(\cF): \hskip 0.2truecm\cF \subseteq \binom{[n]}{r} \text{is maximal intersecting with}\ \cup \cF=[n]\right\}.
\]
The family giving the extremal size in the theorem of Erd\H os, Ko and Rado seems to be a natural candidate for
achieving the value of $M(n,r)$. In fact, most families $\cF$ that occur in the literature have $\cR(\cF)=\Theta(\frac{n}{r})$. In \sref{2} we will prove the following theorems showing that both $M(n,r)$ and $m(n,r)$ have different order of magnitude.
\begin{theorem}
\label{thm:max}

\textbf{(i)} For any values of $n$ and $r$, we have $M(n,r) \le n+r^r$. In particular, if $r < \frac{\log n}{\log\log n}$, then $M(n,r) \le (1-o(1))n$ holds,

\noindent\textbf{(ii)} if $2r+2<n$, then 
\[M(n,r)\ge n-2r+3 -\frac{n-2r+2}{\binom{2r-3}{r-2}} 
\]
holds, in particular if $\omega(1)=r< \frac{\log n}{\log\log n}$, then
we obtain $M(n,r) \sim n$.
\end{theorem}

At first sight, the bound $r^r$ seems to be very weak, but we will show in \sref{3} that it cannot be strengthened too much in general. 

A trivial lower bound on $m(n,r)$ is 1. The next theorem states that $\frac{n}{r^2}$ is another general lower and we construct families showing that this is the order of magnitude of $m(n,r)$ as long as $r=O(n^{1/2})$. For larger values of $r$ we obtain regular maximal families showing the tightness of the trivial lower bound.

\begin{theorem}
\label{thm:min}

\textbf{(i)} For any $n$ and $r$, the inequality $m(n,r) \ge \frac{n}{r^2}$ holds,

\noindent\textbf{(ii)} if $r\le n^{1/2}$, then $m(n,r)=\Theta(\frac{n}{r^2})$ holds,

\noindent\textbf{(iii)} if $\omega(n^{1/2})=r=o(n)$ and $r(n)/n$ is monotone, then there exist infinitely many values $n'$ and $r'=r'(n')$ such that $m(n',r'(n'))=1$ and $r\sim r'$ holds.
\end{theorem}
\section{Proofs}
\label{sec:2}
In this section we prove \tref{max} and \tref{min}.

\begin{proof}[Proof of \tref{max}] To prove \textbf{(i)}, let us consider a maximal intersecting family $\cF \subseteq \binom{[n]}{r}$. Let us partition $\cF$ into two subfamilies $\cF_1$ and $\cF_2$ where $\cF_1:=\{F \in \cF: \exists x \in F \hskip 0.2truecm \text{such that}\ F\setminus \{x\}\cap F'\ne \emptyset \hskip 0.2truecm \text{for all}\ F' \in \cF\}$ and $\cF_2=\cF \setminus\cF_1$. 
\begin{claim}
\label{clm:recursion}
$|\cF_2| \le r^r$.
\end{claim}
\begin{proof}[Proof of Claim] Let $d_j$ denote the maximum number of sets in $\cF_2$ that contain the same $j$-subset. For any $j<r$ and $j$-subset $J$ that is contained in some $F \in \cF_2$ there exists an $F' \in \cF$ with $J \cap F'=\emptyset$. Thus $d_j \le rd_{j+1}$ holds. Since $d_r=1$ and $|\cF_2| \le rd_1$, the claim follows.
\end{proof}
Let $\tau$ denote the covering number of $\cF$, i.e. the minimum size of a set meeting all sets of $\cF$. Clearly, if $\tau=r$, then $\cF_1=\emptyset$ and thus by \clref{recursion} $|\cF| \le r^r$ and $\cR(\cF) \le r^r$. 

Assume $\tau<r$. We will show a roughly $n$ to 1 mapping $f$ from $\cF_1$ to $\cF_{min}$, the subfamily containing one fixed vertex $y$ of minimum degree. For any $F \in \cF_1$ let $g(F)$ be an element of $F$ so that $F \setminus \{g(F)\} \cap F' \ne \emptyset$ for all $F'\in \cF$ (such an element exists by definition of $\cF_1$). Let us define $f(F)=F$ if $y \in F$ and $f(F)=F \setminus \{g(F)\} \cup y$ if $y \notin F$. Note that $f(F)\in \cF$ as already $F \setminus \{g(F)\}$ meets all sets in $\cF$ and by assumption $\cF$ is a maximal intersecting family. Observe that at most $n-r+1$ sets can be mapped to the same set $G$ as all such sets should contain $G \setminus \{y\}$. This concludes the proof of \textbf{(i)} as $\cR(\cF) \le \cR(\cF_1)+|\cF_2|$.

To prove \textbf{(ii)} we need to show a construction. Let us write $S=[2,2r-2]$ and $S_0=[2,r-1]$ and define
\[
\cF_1=\left\{\{1\} \cup G: G \in \binom{S}{r-1}\right\}, \hskip 0.3truecm \cF_2=\left\{\{1,i\}\cup H: 2r-1\le i \le n, H\in \binom{S}{r-2}\setminus S_0\right\}, 
\]
\[
\cF_3=\binom{S}{r}, \hskip 0.3truecm \cF_4=\{(S\setminus S_0) \cup \{i\}: 2r-1 \le i \le n\}, \hskip 0.3truecm \cF=\cup_{j=1}^4\cF_j.
\]
\begin{claim}
\label{clm:maxexample} The family $\cF$ is maximal intersecting.
\end{claim}
\begin{proof}[Proof of Claim]
$\cF$ is clearly intersecting as all of its sets, except those coming from $\cF_2$, meet $S$ in at least $r-1$ elements. A set $F_2$ from $\cF_2$ meets any other from $\cF_1 \cup \cF_2$ as they both contain $1$, a set from $\cF_3$ because of the pigeon-hole principle, and a set from $\cF_4$ as by definition $F_2 \cap S \ne S_0$.

To prove the maximality of $\cF$ let us consider a set $T\notin \cF$. If $|T \cap S|<r-2$, then any $r$-subset of $S\setminus T$ is in $\cF$ and thus $T$ cannot be added to $\cF$. As all $r$-subsets of $S$ are already in $\cF$, it remains to deal with
the cases $|T \cap S|=r-1$ and $|T \cap S|=r-2$. If $|T\cap S|=r-1$, then $1 \notin T$ as those sets are in $\cF_1$ and $T \cap S\ne S \setminus S_0$ as those sets are in $\cF_4$. But then a set $F$ from $\cF_2$ with $F\cap S=S\setminus T$ is disjoint from $T$, thus $T$ cannot be added to $\cF$.

Finally, if $|T\cap S|=r-2$, then if $1 \notin T$, then $\{1\} \cup (S \setminus T) \in \cF_1$ is disjoint from $T$ and thus $T$ cannot be added to $\cF$. If $1 \in T$, then, as $T \notin \cF_2$, we must have $T\cap S=S_0$. Then we can find a set disjoint from $T$ in $\cF_4$.
\end{proof}

All we have to observe is that in $\cF$ the degree of 1 is $\binom{2r-3}{r-1}+(\binom{2r-3}{r-2}-1)(n-2r+2)$ and the degree of $i$ for any $2r-1\le i \le n$ is $\binom{2r-3}{r-2}$.
\end{proof}

Note that the proof of \tref{max} \textbf{(i)} gives an upper bound $M(n,r) \le n+r^r$ for any value of $r$ and $n$.

\begin{conjecture}
\label{conj:bigmax} If $r=o(n)$ holds, then the order of magnitude of $M(n,r)$ is $\Theta(n)$.
\end{conjecture}

\vskip 0.5truecm

Now we turn our attention to the function $m(n,r)$. In the proof of \tref{min} we will use the following theorem by Blokhuis on blocking sets of projective planes (for a short survey on the topic see \cite{Sz}).

\begin{theorem}[Blokhuis, \cite{B}]
\label{thm:prplane}
Let $Q$ be a projective plane of order $q$ and $B$ be a blocking set (a set that meets all lines of the projective plane) of size less than $\frac{3}{2}(q+1)$. If $q$ is prime, then $B$ contains a line of the projective plane.
\end{theorem}

We will also need the following strengthening of Chebyshev's theorem.

\begin{theorem}[Nagura, \cite{N}]
\label{thm:prime} For every integer $n \ge 25$ there exists a prime $p$ with $n \le p \le (1+1/5)n$.
\end{theorem}

\begin{proof}[Proof of \tref{min}]
To prove \textbf{(i)} we make the following two easy observations: for any intersecting family $\cF$ we have $\Delta(\cF)\ge |\cF|/r$ as for any set $F \in \cF$ the inequality $\sum_{x \in F}d(x) \ge |\cF|$ holds. Also, the average degree in $\cF$ equals $\frac{r|\cF|}{n}$. As the average degree is at least as large as the minimum degree, we obtain
\[
\cR(\cF)=\frac{\Delta(\cF)}{\delta(\cF)} \ge \frac{\frac{|\cF|}{r}}{\frac{r|\cF|}{n}}=\frac{n}{r^2}.
\]
Note that the proof does not use the fact that $\cF$ is maximal.

To prove \textbf{(ii)} and \textbf{(iii)} we need constructions. Suppose first that $r \le n^{1/2}$ holds. By \tref{prime} we can pick a prime $p$ such that $\frac{2}{3}r<p<\frac{2}{3}(1+\frac{1}{5})r=\frac{4}{5}r$. Let $P$ denote a projective plane of order $p$ with vertex set $[p^2+p+1]$. Let us define the following maximal intersecting family
\[
\cF_{n,r,p}=\left\{F \in \binom{[n]}{r}: l \subset F \hskip 0.3truecm \text{for some line}\ l \in P\right\}.
\]
Note that $\cF_{n,r,p}$ is intersecting as any two of its sets intersect as they both contain lines of a projective plane and 
$\cF_{n,r,p}$ is maximal as if $G \in \binom{[n]}{r}$ does not contain any line of $P$, then by \tref{prplane} and $r <\frac{3p}{2}$ we know that there exists a line $l$ in $P$ such that $l \cap G=\emptyset$ and this line can be extended to a set
$F_l \supset l$ such that $F_l \cap G=\emptyset$ holds. As every vertex is contained in $p+1$ lines of $P$ we have that $d(x)=(p+1)\binom{n-p-1}{r-p-1}+p^2\binom{n-p-2}{r-p-2}$ if $x \in [p^2+p+1]$, while for any $y \in [p^2+p+2,n]$ we have $d(y)=(p^2+p+1)\binom{n-p-2}{r-p-2}$. Therefore we obtain
\[
\cR(\cF_{n,p,r})=\frac{p^2\binom{n-p-2}{r-p-2}+(p+1)\binom{n-p-1}{r-p-1}}{(p^2+p+1)\binom{n-p-2}{r-p-2}}\le1+ \frac{1}{p+1}\cdot\frac{n-p-1}{r-p-1} \le \frac{25}{4}\cdot \frac{n}{r^2}
\]
where the last inequality follows from $p \le \frac{4}{5}r$. Note that by the prime number theorem one can pick $p$ such that $p\sim \frac{2}{3}r$ provided that $r$ is large enough, and thus improve the constant $\frac{25}{4}$ to $\frac{9}{4}$.

It remains to prove \textbf{(iii)}. Consider the following general construction $\cF'_{k,p,s}\subseteq \binom{[n]}{r}$ where $1 \le k$ is an odd integer, $p$ is a prime, $0\le s \le \frac{p}{2}$ and $n=k(p^2+p+1)$, $r=\frac{k+1}{2}(p+1)+s$. For $1 \le i \le k$ let $P_i$ be a projective plane of order $p$ with underlying set $[(i-1)(p^2+p+1)+1,i(p^2+p+1)]$ and let us define
\[
\cF'_{k,p,s}=\left\{F \in \binom{[n]}{r}: F \hskip 0.2truecm \text{contains a line of}\ P_i \hskip 0.2truecm \text{if}\ i \in I \hskip0.2truecm \text{for some}\ I\in \binom{[k]}{\frac{k+1}{2}}\right\}.
\]
As any two lines of a projective plane intersect each other and so do any $I,I' \in \binom{[k]}{\frac{k+1}{2}}$, the family $\cF'_{k,p,s}$ is intersecting and by \tref{prplane} and $s\le \frac{p}{2}$ we obtain the maximality of $\cF'_{k,p,s}$. As the construction is symmetric, all degrees are equal and therefore we obtain $\cR(\cF'_{k,p,s})=1$.

Assume that we are given a sequence of integers $r=r(n)$ with $r =\omega(n^{1/2})$. Let us pick a prime $p$ with $p\sim \frac{n}{2r}$ and an odd integer $k\sim \frac{4r^2}{n}$. Then we can consider the family $\cF_{k,p,s}$ with any $0 \le s \le p/2$. Its vertex has size $k(p^2+p+1) =n'\sim n$ and by the monotonicity of $r/n$ and $r=\omega(n^{1/2})$ we obtain that the sets of $\cF'_{k,p,s}$ have size $\frac{k+1}{2}(p+1)+s=r' \sim r$.
\end{proof}

\vskip 0.5truecm

\section{Concluding remarks}
\label{sec:3}
As we mentioned in the Introduction, the bound of \tref{max} \textbf{(i)} cannot be strengthened in general as the following example shows. If $n=2r$, then a maximal intersecting family $\cF$ contains one set from every pair of complement sets. Thus the family $\cF^*=\{F \in \binom{[n]}{r}: 1 \notin F, F \neq [r+1,n]\} \cup \{[r]\}$ is maximal intersecting and $\cR(\cF^*)=\Theta(\binom{n}{r})=e^{\Theta(n)}$ holds while $r^r=e^{\Theta(n\log n)}$.

In \tref{min} \textbf{(iii)}, we could show regular maximal intersecting families only for special values of $n$ and $r$. There are two ways to generalize our construction. First, one needs not insist that all projective planes should be of the same order, but for the maximality one still needs that they should be of the same asymptotic order (one will have to choose $s$ a bit more carefully). This will ruin the regularity, but for families $\cF$ obtained this way $\cR(\cF)=1+o(1)$ would still hold. The other possibility is to add extra vertices that do not belong to $\cup P_i$, similarly to the construction used for \tref{min} \textbf{(iii)}. This will enable us to obtain constructions for arbitrary values of $n$ and $r$ (provided $n$ is large enough) but for these families $\cF'$ we will have $\cR(\cF')=\Theta(\frac{r^2}{n})$.

It remains open whether one can construct maximal intersecting families with $\cR$-value $1+o(1)$ for any $r(n)$.

\vskip 0.5truecm

\noindent \textbf{Acknowledgement.} This research was started at the $3^{rd}$ Eml\'ekt\'abla Workshop held in Balatonalm\'adi, June 27-30, 2011.

\end{document}